\documentclass[12pt,reqno]{amsart}

\usepackage[a4paper,margin=30mm]{geometry}
\usepackage{amsmath,amssymb,amsthm,mathtools}
\numberwithin{equation}{section}
\usepackage{bm}
\usepackage{enumitem}
\usepackage[hidelinks]{hyperref}
\hypersetup{
 pdftitle={Global finite-energy weak solutions with vacuum for a periodic 1D Navier--Stokes--Maxwell system},
 pdfauthor={Cheng Yu},
 pdfsubject={Global finite-energy weak solutions with vacuum for a periodic compressible Navier--Stokes--Maxwell system},
 pdfkeywords={Maxwell--Ohm equations, bilinear-control continuity, weak-to-strong continuity, compressible Navier--Stokes--Maxwell system, finite-energy weak solutions, vacuum}
}

\usepackage{aliascnt}
\newtheorem{theorem}{Theorem}[section]
\newaliascnt{proposition}{theorem}
\newtheorem{proposition}[proposition]{Proposition}
\aliascntresetthe{proposition}
\newaliascnt{lemma}{theorem}
\newtheorem{lemma}[lemma]{Lemma}
\aliascntresetthe{lemma}
\newaliascnt{corollary}{theorem}

\aliascntresetthe{corollary}
\theoremstyle{definition}
\newaliascnt{definition}{theorem}
\newtheorem{definition}[definition]{Definition}
\aliascntresetthe{definition}
\newaliascnt{remark}{theorem}
\newtheorem{remark}[remark]{Remark}
\aliascntresetthe{remark}
\usepackage[nameinlink,capitalize,noabbrev]{cleveref}
\crefname{proposition}{Proposition}{Propositions}
\Crefname{proposition}{Proposition}{Propositions}
\crefname{lemma}{Lemma}{Lemmas}
\Crefname{lemma}{Lemma}{Lemmas}
\crefname{corollary}{Corollary}{Corollaries}
\Crefname{corollary}{Corollary}{Corollaries}
\crefname{definition}{Definition}{Definitions}
\Crefname{definition}{Definition}{Definitions}
\crefname{remark}{Remark}{Remarks}
\Crefname{remark}{Remark}{Remarks}

\newcommand{\T}{\mathbb T}
\newcommand{\R}{\mathbb R}
\newcommand{\D}{\mathcal D}

\newcommand{\weak}{\rightharpoonup}
\newcommand{\weakstar}{\stackrel{*}{\rightharpoonup}}
\newcommand{\eps}{\varepsilon}

\newcommand{\meanx}[1]{\langle #1\rangle_x}

\title[Finite-energy solutions with vacuum]{Global finite-energy weak solutions to the Navier--Stokes--Maxwell in 1D}
\author{Cheng Yu}
\address{Department of Mathematics, University of Florida, Gainesville, FL 32611, USA}
\email{chengyu@ufl.edu}
\date{}

\begin{document}

\begin{abstract}
We construct global finite-energy weak solutions for a 1D barotropic Navier--Stokes--Maxwell system with displacement current and the algebraic Ohm law. The result holds for every \(\gamma>1\) and arbitrary finite-energy initial data, allowing vacuum and requiring only \(L^2\) initial electromagnetic fields.

A key ingredient is a weak-to-strong compactness principle for the Maxwell--Ohm subsystem: weak convergence of the velocity coefficients in \(L^2_tH^1_x\), together with strong convergence of the initial fields, yields strong electromagnetic trajectories in \(C_tL^2_x\). This permits identification of the weak current and the distributional Lorentz force without strong convergence of the velocity. The electromagnetic closure is combined with the classical artificial-viscosity and artificial-pressure construction for compressible flow. A density-primitive effective-flux argument identifies the physical pressure and eliminates the physical and artificial pressure-concentration defects.
\end{abstract}


\maketitle

\section{Introduction}

The Navier--Stokes--Maxwell system couples a viscous conducting fluid to an
electromagnetic field propagated by Maxwell's equations.  Keeping the
displacement current makes the electromagnetic component genuinely
hyperbolic.  This is analytically different from replacing Maxwell's
equations by a quasistatic or magnetohydrodynamic approximation: an energy
bound controls the electric and magnetic fields in $L^2$, but does not by
itself make either field compact.  Consequently, the Lorentz force contains
precisely the kind of weak--weak product that is invisible to the usual energy
method.

The same obstruction is central in the higher-dimensional incompressible
Navier--Stokes--Maxwell theory. Masmoudi~\cite{Masmoudi2010} proved global
well-posedness in two dimensions for
$u_0\in L^2(\mathbb R^2)$ and
$E_0,B_0\in H^s(\mathbb R^2)$, $0<s<1$, allowing an arbitrarily small
amount of electromagnetic regularity beyond the energy space.
Ibrahim and Keraani~\cite{IbrahimKeraani2011} constructed global small
solutions, while Germain, Ibrahim and
Masmoudi~\cite{GermainIbrahimMasmoudi2014} obtained local solutions for
large data and global solutions for small data in nearly critical spaces.
Further global results include the large-energy theory of Ars\'enio and
Gallagher~\cite{ArsenioGallagher2020} and the axisymmetric
three-dimensional theory of Ars\'enio, Hassainia and
Houamed~\cite{ArsenioHassainiaHouamed2024}.
Ars\'enio, Houamed and
Said--Houari~\cite{ArsenioHouamedSaidHouari2025} extended the
two-dimensional theory to inhomogeneous fluids under suitable assumptions
on the density, velocity and electromagnetic fields.
These developments do not resolve the compactness obstruction at the
natural energy level: the energy estimate alone controls the electromagnetic
fields in $L^\infty_tL^2_x$ but does not provide the compactness needed to
pass to the limit in the Lorentz force.
This obstruction is also highlighted in the recent discussion of
Houamed, Ibrahim and
Said--Houari~\cite{HouamedIbrahimSaidHouari2026}.
In particular, global weak solutions for arbitrary data in the natural
energy space remain open in both two and three dimensions, a problem
already emphasized by Masmoudi~\cite{Masmoudi2010}; see also the monograph
of Ars\'enio and Saint-Raymond~\cite{ArsenioSaintRaymond2019}.

The 1D compressible system considered here is not a dimensional reduction
of that incompressible problem.  Rather, it is a model problem that retains
the displacement current, the algebraic Ohm law, and the same weak--weak
electromagnetic coupling.  In this setting, the corresponding natural-energy
compactness obstruction can be overcome: the Maxwell--Ohm subsystem enjoys
the weak-to-strong compactness property of Theorem~1.1, which provides the
electromagnetic closure needed for the global finite-energy weak existence
result of Theorem~1.4, including vacuum.
The electromagnetic closure relies on strong compactness at fixed spatial frequency and a uniform removal of the spatial cutoff using the 1D regularity of the velocity coefficients.

Let $\T=\R/\mathbb Z$, normalized by $|\T|=1$, and let
$a,\mu,\eps,\sigma>0$.  We consider
\begin{align}
 \rho_t+(\rho u)_x&=0,                                      \label{eq:mass}\\
 (\rho u)_t+(\rho u^2+a\rho^\gamma)_x-\mu u_{xx}&=-jb,       \label{eq:mom}\\
 \eps E_t-b_x&=-j,                                          \label{eq:ampere}\\
 b_t-E_x&=0,                                                 \label{eq:faraday}\\
 j&=\sigma(E+ub),                                            \label{eq:ohm}
\end{align}
supplemented by the periodic initial data
\[
 \rho(0,\cdot)=\rho_0,\qquad (\rho u)(0,\cdot)=m_0,\qquad
 E(0,\cdot)=E_0,\qquad b(0,\cdot)=b_0.
\]
Here $\gamma>1$; $\rho$ is the fluid mass density, $u$ is the
longitudinal fluid velocity, $E$ and $b$ are the scalar transverse electric
and magnetic field components, respectively, and $j$ is the corresponding
transverse electric current density.  The barotropic pressure is
$p(\rho)=a\rho^\gamma$; thus $a$ and $\gamma$ are the pressure parameters,
$\mu$ is the viscosity coefficient, $\eps$ is the electric permittivity
(the coefficient of the displacement current), and $\sigma$ is the electrical
conductivity.

This is the scalar reduction obtained from $u=(u,0,0)$, $E=(0,0,E)$ and
$B=(0,b,0)$: the current is $j=(0,0,j)$ and
$j\times B=(-jb,0,0)$.  In particular, the sign in \eqref{eq:mom} is fixed by
the three-dimensional Lorentz force.  The Gauss-law constraints are satisfied
identically in this transverse reduction.  The unknown $b$ is the transverse
magnetic component, not a magnetic potential.

The formal total energy is
\begin{equation}
 \mathcal E(t)=\int_\T\left(
 \frac12\rho u^2+\frac{a}{\gamma-1}\rho^\gamma
 +\frac\eps2E^2+\frac12b^2\right)\,dx.                       \label{eq:energy}
\end{equation}
For smooth solutions,
\begin{equation}
 \frac{d}{dt}\mathcal E(t)
 +\mu\int_\T |u_x|^2\,dx
 +\frac1\sigma\int_\T |j|^2\,dx=0.                         \label{eq:energy-id}
\end{equation}
Indeed, the fluid and electromagnetic coupling terms add up to
\(
 -j(ub+E)=-|j|^2/\sigma.
\)

The natural estimates following from \eqref{eq:energy-id} are
\begin{equation}
 \rho\in L^\infty_tL^\gamma_x,\quad
 \sqrt\rho\,u\in L^\infty_tL^2_x,\quad
 u_x,j\in L^2_{t,x},\quad
 E,b\in L^\infty_tL^2_x.                                  \label{eq:natural-bounds}
\end{equation}
The positive total mass controls the missing spatial mean of $u$, so  $u\in L^2_tH^1_x\hookrightarrow L^2_tL^\infty_x$ in 1D.  Nevertheless,
these bounds give only weak convergence of $u$, $b$ and $j$ along an
approximation.  Neither $ub$ in Ohm's law nor $jb$ in the momentum equation is
therefore identified by a direct compactness argument.

\subsection{Maxwell--Ohm closure via weak-control continuity.}

\begin{theorem}[Weak-to-strong Maxwell--Ohm compactness]\label{thm:maxwell}
Let $v_n\weak v$ in $L^2(0,T;H^1(\T))$, and let
$(E_{0,n},b_{0,n})\to(E_0,b_0)$ in $L^2(\T)^2$.  For each $n$, let
$(E_n,b_n)$ be the energy solution of
\begin{equation}
 \eps E_{n,t}-b_{n,x}=-\sigma(E_n+v_nb_n),\qquad
 b_{n,t}-E_{n,x}=0.                                        \label{eq:intro-maxwell}
\end{equation}
with initial datum $(E_{0,n},b_{0,n})$.  Let $(E,b)$ denote the unique energy
solution with coefficient $v$ and initial datum $(E_0,b_0)$.  Then the entire
sequence satisfies
\begin{equation}
 (E_n,b_n)\to(E,b)\quad\text{in }C([0,T];L^2(\T)^2).        \label{eq:maxwell-strong}
\end{equation}
Moreover, with $j_n=\sigma(E_n+v_nb_n)$,
\begin{equation}
 j_n\weak j=\sigma(E+vb)\quad\text{in }L^2_{t,x},\qquad
 j_nb_n\to jb\quad\text{in }\D'((0,T)\times\T).          \label{eq:lorentz-stable}
\end{equation}
Thus the Maxwell--Ohm solution map is weak-to-strong continuous with respect
to its velocity coefficient.
\end{theorem}

The energy bounds give no spatial derivative control of $E_n$ or $b_n$,
so the usual Aubin--Lions--Simon argument~\cite{Aubin1963,Simon1987} does
not directly provide this closure.  
At a fixed spatial cutoff, the resulting finite-input bilinear evolution can be treated by classical weak-control continuity 
arguments; see, e.g., \cite{BallMarsdenSlemrod1982,Dirr2023}. We give the finite-sum argument needed here directly in \cref{sec:maxwell}.

Here a fixed spatial cutoff of the velocity produces finitely many scalar
time-dependent controls acting on the full infinite-dimensional Maxwell
energy space.  Removing that cutoff still requires the 1D
Fourier-tail estimate and stability bounds uniform in the cutoff.  The
strong initial-data hypothesis is also retained: this argument asserts
neither smoothing nor compactness for arbitrary bounded initial fields.
After strong field convergence, separate weak--strong product arguments give
the current and Lorentz-force limits.  Thus \cref{thm:maxwell} supplies the
PDE closure needed by the fluid approximation without requiring strong
velocity convergence.

\subsection{The global finite-energy existence result.}
The main result, \cref{thm:main}, constructs global finite-energy
weak solutions of
\eqref{eq:mass}--\eqref{eq:ohm} for every $\gamma>1$ and arbitrary
finite-energy initial data, including data with local vacuum.
The result retains both the displacement current and the algebraic Ohm law.

For the fluid part we use the 1D density-primitive version of
the classical inverse-divergence and effective-flux framework
\cite{Lions1998,FeireislNovotnyPetzeltova2001}.  These references supply the
compressible-fluid framework; their multidimensional existence theorems do
not directly establish the present 1D coupled result for every
$\gamma>1$.  On the torus the test is the mean-zero density primitive
\begin{equation}
 w=\partial_x^{-1}(\rho-M),\qquad
 M=\int_\T\rho\,dx>0.                                      \label{eq:w-def}
\end{equation}
Here and below, for every integrable periodic function $f$, we write
\[
 \meanx{f}:=\int_\T f(x)\,dx
\]
for its spatial mean; recall that $|\T|=1$.  The continuity equation then gives
$w_t=-\rho u+\meanx{\rho u}$.  Energy bounds therefore make $w$ compact in
space-time.  In the present coupled problem, its
strong compactness combines with the Maxwell--Ohm theorem to identify the
electromagnetic contribution involving $jbw$.  Testing the momentum equation
by $w$ then cancels two apparently nonintegrable copies of $\rho^2u^2$ before
either one is estimated and yields the gain
$\rho\in L^{\gamma+1}_{t,x}$.  The same identity provides the
effective-viscous-flux relation needed to identify the pressure.  What is
specific to the present problem is the combination of this classical fluid
test with weak-coefficient Maxwell compactness, which closes the full
electromagnetic-fluid limit.

The proof combines the following interdependent parts within the classical
artificial-viscosity and artificial-pressure scheme: the exact Maxwell
evolution at the Galerkin level, electromagnetic closure in the fluid limits,
the Lorentz force against moving density primitives, the uniform density gain,
and elimination of the pressure defects.  In the final pressure limit we
retain every possible physical and artificial pressure concentration as a
nonnegative measure. The density-primitive effective-flux identity, together with the
renormalized $\rho\log\rho$ identity, then forces all defects to vanish.
  The remaining
Lions--Feireisl compactness and renormalization steps are cited in their
standard form.

The uncoupled 1D compressible Navier--Stokes equations have a
classical large-data and nonsmooth-data theory; see, in particular,
Hoff~\cite{Hoff1986,Hoff1987}.  Those results do not contain the hyperbolic
weak--weak product that is the closure issue here.  For compressible
Navier--Stokes--Maxwell systems, local well-posedness with vacuum is studied
by Fan and Jia~\cite{FanJia2018} and Fan, Jing, Nakamura and
Tang~\cite{FanJingNakamuraTang2020}.  Fan and Hu~\cite{FanHu2014} study the
1D complete equations for an electromagnetic fluid.  Wave
stability is considered in
\cite{LuoYaoZhu2021,YaoZhu2021,YaoZhu2022}.  Hou, Yao and
Zhu~\cite[Theorem~2.2]{HouYaoZhu2016} establish global strong solutions with
large initial data and vacuum for a 1D initial-boundary problem
on $(0,1)$.  Their model includes a longitudinal Gauss constraint and the
transverse Amp\`ere source $\rho w$, where $w\in\R^2$ denotes their transverse
velocity; see \cite[equations~(1.5)--(1.6)]{HouYaoZhu2016}.  It does not impose
an algebraic Ohm law.  Their initial data have Sobolev regularity and satisfy
the compatibility conditions (1.9) of that paper.  The global theorem covers
$\gamma\ge2$, and general $\gamma>1$ under an additional boundary-trace
integral bound or sign condition.  It therefore does not contain the periodic
natural-energy existence theorem below.  The energy-equality result of Zhang, Huang and
Wu~\cite{ZhangHuangWu2023} concerns criteria for equality for given weak
solutions.  The result proved here concerns the precise periodic barotropic
system \eqref{eq:mass}--\eqref{eq:ohm}, with its displacement current and
algebraic Ohm law, in the initial-data and solution classes specified below.

Our fluid approximation follows the artificial-pressure and
artificial-viscosity architecture of the compressible theory
\cite{Lions1998,FeireislNovotnyPetzeltova2001}.  We keep the Maxwell--Ohm system exact throughout the fluid approximation.
This choice preserves its energy cancellation and permits direct application
of \cref{thm:maxwell} to the full hyperbolic evolution at each limit.
Higher-dimensional incompressible Navier--Stokes--Maxwell systems exhibit a
related energy-space compactness issue; see, for example,
\cite{ArsenioGallagher2020}.  The argument below is genuinely
1D and makes no claim for those systems.

\subsection{Finite-energy formulation and main existence theorem}
\label{subsec:weak-main}

Throughout the paper, the letter $M$ is reserved for the conserved total
mass.  The Fourier bandwidth in the compactness argument is denoted by
$\Lambda$. Let $p_0=2\gamma/(\gamma+1)>1$.  We use the conventions
\begin{equation*}|m_0|^2/\rho_0=0\;\text{ on }  \{\rho_0=0,m_0=0\} \text{ and }\,\, +\infty \,\text{ on }
\{\rho_0=0,m_0\ne0\}. \end{equation*}

\begin{definition}[Finite-energy weak solution]\label{def:weak}
Let $T>0$.  A quintuple $(\rho,u,E,b,j)$ is a finite-energy weak solution of
\eqref{eq:mass}--\eqref{eq:ohm} on $(0,T)\times\T$ if $\rho\ge0$,
$m=\rho u$, and
\begin{align*}
 &\rho\in C_w([0,T];L^\gamma(\T))\cap L^{\gamma+1}((0,T)\times\T),\\
 &m\in C_w([0,T];L^{p_0}(\T)),\qquad
 \sqrt\rho\,u\in L^\infty(0,T;L^2(\T)),\\
 &u\in L^2(0,T;H^1(\T)),\qquad
 E,b\in C([0,T];L^2(\T)),\qquad j\in L^2((0,T)\times\T),
\end{align*}
$j=\sigma(E+ub)$ almost everywhere, and the following identities hold for
smooth periodic tests vanishing at $t=T$:
\begin{align}
 &\int_0^T\!\!\int_\T(\rho\phi_t+\rho u\phi_x)\,dx\,dt
   +\int_\T\rho_0\phi(0)\,dx=0,                            \label{eq:weak-mass}\\
 &\int_0^T\!\!\int_\T\left[
 m\varphi_t+(mu+a\rho^\gamma)\varphi_x
 -\mu u_x\varphi_x-jb\varphi\right]\,dx\,dt
 +\int_\T m_0\varphi(0)\,dx=0,                            \label{eq:weak-mom}\\
 &\int_0^T\!\!\int_\T(\eps E\psi_t-b\psi_x-j\psi)\,dx\,dt
 +\eps\int_\T E_0\psi(0)\,dx=0,                           \label{eq:weak-E}\\
 &\int_0^T\!\!\int_\T(b\chi_t-E\chi_x)\,dx\,dt
 +\int_\T b_0\chi(0)\,dx=0.                              \label{eq:weak-b}
\end{align}
The continuity equation is renormalized: for every $B\in C^1([0,\infty))$
with compactly supported derivative,
\begin{equation}
 \partial_tB(\rho)+\partial_x(B(\rho)u)
 +[B'(\rho)\rho-B(\rho)]u_x=0                             \label{eq:renorm}
\end{equation}
in distributions.  Finally, for almost every $t\in(0,T)$,
\begin{align}
 &\mathcal E(t)+\int_0^t\!\!\int_\T
 \left(\mu|u_x|^2+\frac1\sigma|j|^2\right)\,dx\,ds
 \le \mathcal E_0,                                        \label{eq:energy-ineq}\\
 &\mathcal E_0=\int_\T\left[
 \frac12\frac{|m_0|^2}{\rho_0}
 +\frac{a}{\gamma-1}\rho_0^\gamma
 +\frac\eps2E_0^2+\frac12b_0^2\right]dx.                 \nonumber
\end{align}
\end{definition}

\begin{remark}[Nonlinear terms and velocity on vacuum]\label{rem:products}
All products in \cref{def:weak} are integrable at the stated regularity.  In
particular, $\rho u^2\in L^\infty_tL^1_x$ by the kinetic energy,
$\rho^\gamma\in L^{(\gamma+1)/\gamma}_{t,x}$ by the additional density gain,
and
\[
 ub\in L^2_{t,x},\qquad jb\in L^1_{t,x},
\]
because $H^1(\T)\hookrightarrow L^\infty(\T)$,
$u\in L^2_tH^1_x$, $b\in L^\infty_tL^2_x$, and $j\in L^2_{t,x}$.  Thus Ohm's
law is an almost-everywhere identity and the Lorentz force in
\eqref{eq:weak-mom} is an ordinary $L^1$ function.  The weak temporal
continuity of $\rho$ and $m$ follows from their distributional equations and
the indicated uniform spatial bounds.

The velocity $u$ in \cref{def:weak} is a genuine component of the solution,
not merely an abbreviation for $m/\rho$.  The identity $m=\rho u$ determines
$u$ almost everywhere on $\{\rho>0\}$, but the fluid pair $(\rho,m)$ alone
does not determine its $H^1$ extension on a vacuum set of positive measure.
Such an extension is not immaterial.  The unweighted viscous stress $\mu u_x$
already depends on it, and in the coupled system it also enters Ohm's law
$j=\sigma(E+ub)$.  Consequently, changing $u$ on a vacuum region while
keeping $(\rho,m)$ fixed will generally change the current, the Maxwell
evolution, and the Lorentz force, and need not preserve the momentum equation.

Accordingly, \cref{thm:main} is an existence theorem for at least one
quintuple $(\rho,u,E,b,j)$; it does not assert that $u$, and hence the
electromagnetic fields, are determined by $(\rho,m)$ alone, nor does it provide
a canonical selection from the initial data.  In the construction below, the
smooth positive-density approximations produce
velocities that converge, after extraction, weakly in $L^2_tH^1_x$ to the
velocity appearing in the limiting quintuple.  This is the extension selected
by the chosen approximating subsequence.  \end{remark}

The following global existence theorem is the main result of the paper;
\cref{thm:maxwell} supplies its electromagnetic closure.

\begin{theorem}[Global finite-energy existence with vacuum]\label{thm:main}
Let $\gamma>1$ and $a,\mu,\eps,\sigma>0$.  Assume
\begin{equation}
 \rho_0\ge0,\quad \rho_0\in L^\gamma(\T),\quad
 M:=\int_\T\rho_0\,dx>0,                                  \label{eq:data-rho}
\end{equation}
\begin{equation}
 m_0=0\ \text{a.e. on }\{\rho_0=0\},\qquad
 \frac{|m_0|^2}{\rho_0}\in L^1(\T),\qquad E_0,b_0\in L^2(\T).
                                                                    \label{eq:data-rest}
\end{equation}
Then \eqref{eq:mass}--\eqref{eq:ohm} possesses at least one global
finite-energy weak solution $(\rho,u,E,b,j)$ in the sense of
\cref{def:weak}.  Local vacuum is allowed, with $u$ understood as part of the
constructed solution as explained in \cref{rem:products}.  In addition, for
every finite $T$,
\begin{equation}
 \int_0^T\!\!\int_\T\rho^{\gamma+1}\,dx\,dt\le C_T.       \label{eq:main-gain}
\end{equation}
\end{theorem}

The remainder of the paper is organized as follows.
\Cref{sec:primitive} establishes the weighted Poincar\'e estimate, compactness
of the density primitive and the pressure-gain identity.
\Cref{sec:maxwell}  proves the Maxwell--Ohm closure theorem by a fixed-cutoff argument and uniform coefficient approximation.
\Cref{sec:application} completes the
classical finite-energy construction, emphasizing only the coupling-specific
compactness and pressure-defect arguments.

\section{Energy consequences and the density primitive}
\label{sec:primitive}

The next elementary estimate is the reason that positive total mass,
rather than a pointwise lower density bound, controls the unweighted velocity.
If $\rho\ge0$, $\int_\T\rho=M>0$, $\sqrt\rho\,u\in L^2$ and $u\in H^1$, then
\begin{equation}
 \|u\|_{H^1}\le C\left(\|u_x\|_2+M^{-1/2}\|\sqrt\rho\,u\|_2\right).
 \label{lem:weighted-poincare}
\end{equation}
Indeed, let $\meanx{u}=\int_\T u$.  Since
$\|u-\meanx{u}\|_\infty\le C\|u_x\|_2$,
\[
 M|\meanx{u}|\le\left|\int\rho u\right|+
 \int\rho|u-\meanx{u}|
 \le M^{1/2}\|\sqrt\rho u\|_2+CM\|u_x\|_2.
\]
The ordinary Poincar\'e inequality gives the estimate.

For a mean-zero periodic distribution $f$, let $\partial_x^{-1}f$ denote its
mean-zero primitive.  This is the periodic 1D realization of the
inverse-divergence operator used in the standard pressure and effective-flux
estimates for compressible Navier--Stokes equations
\cite{Lions1998,FeireislNovotnyPetzeltova2001}.  If a sequence satisfies the
energy bounds and the continuity equation, define
\begin{equation}
 w_n=\partial_x^{-1}(\rho_n-M_n),\qquad M_n=\int_\T\rho_n.
\end{equation}
Then
\begin{equation}
 (w_n)_t=-\rho_nu_n+\meanx{\rho_nu_n}.                     \label{eq:w-time}
\end{equation}

The following elementary 1D compactness consequence of the continuity
equation will be used to handle the Lorentz force against the moving density
primitive.

\begin{lemma}[Compact primitive]\label{lem:primitive}
Assume $\gamma>1$, $M_n\to M>0$,
$\rho_n$ is bounded in $L^\infty_tL^\gamma_x$, and $u_n$ is bounded in
$L^2_tH^1_x$.  If the continuity equation holds, then, after extraction,
there exists $\rho\in L^\infty(0,T;L^\gamma(\T))$ such that
$\rho_n\weakstar\rho$ in $L^\infty_tL^\gamma_x$ and
\begin{equation}
 w_n\longrightarrow w=\partial_x^{-1}(\rho-M)
 \quad\text{in }C([0,T]\times\T).                          \label{eq:w-compact}
\end{equation}
\end{lemma}

\begin{proof}
The normalization $\int_\T w_n=0$ and $(w_n)_x=\rho_n-M_n$ give
\[
 \sup_n\|w_n\|_{L^\infty_tW^{1,\gamma}_x}<\infty.
\]
Moreover,
\[
 \|\rho_nu_n\|_{L^2_tL^\gamma_x}
 \le \|\rho_n\|_{L^\infty_tL^\gamma_x}
     \|u_n\|_{L^2_tL^\infty_x}.
\]
Thus \eqref{eq:w-time} gives a uniform $1/2$-H\"older modulus in time with
values in $L^\gamma$:
\[
 \|w_n(t)-w_n(s)\|_{L^\gamma}
 \le C|t-s|^{1/2}.
\]
The 1D Gagliardo--Nirenberg inequality, applied to
$w_n(t)-w_n(s)$ and combined with the uniform $W^{1,\gamma}$ bound, upgrades
this to a common modulus of continuity with values in $C(\T)$.  At each fixed
time, $W^{1,\gamma}(\T)\Subset C(\T)$ because $\gamma>1$.  Arzel\`a--Ascoli
therefore gives \eqref{eq:w-compact}.  Finally, weak convergence of the
densities and $(w_n)_x=\rho_n-M_n$ identify the distributional derivative of
the limit, while the zero-mean normalization identifies $w$ itself.
\end{proof}

The next calculation is first stated for smooth solutions.  Its regularized
and low-regularity versions are justified in \cref{sec:application}.

\begin{lemma}[Density-primitive identity]\label{lem:primitive-identity}
Let $(\rho,u,E,b,j)$ be a smooth solution on $[0,T]\times\T$, let
$m=\rho u$, and suppose the pressure $p:[0,\infty)\to[0,\infty)$ is
nondecreasing.  With $w$ defined by \eqref{eq:w-def}, one has
\begin{align}
 \int_0^T\!\!\int_\T p(\rho)(\rho-M)
 ={}&\left[\int_\T mw\right]_0^T
 +\int_0^T\left(M\int_\T\rho u^2-|\meanx{m}|^2\right)dt   \label{eq:primitive-id}\\
 &+\mu\int_0^T\!\!\int_\T u_x(\rho-M)
 +\int_0^T\!\!\int_\T jbw.                              \nonumber
\end{align}
For $p(\rho)=a\rho^\gamma$, the energy bound implies
\begin{equation}
 \int_0^T\!\!\int_\T\rho^{\gamma+1}\,dx\,dt\le C_T.       \label{eq:pressure-gain}
\end{equation}
The same conclusion holds for a pressure containing additional nonnegative
monomials, with the corresponding higher powers on the left-hand side.
\end{lemma}

\begin{proof}
The continuity equation yields $w_t=-m+\meanx{m}$, where
$\meanx{m}(t)=\int_\T m(t,x)\,dx$.  Multiply \eqref{eq:mom} by $w$, integrate in
space and time, and use $w_x=\rho-M$.  The time derivative gives
\[
 \int_0^T\!\!\int_\T m_t w
 =\left[\int_\T mw\right]_0^T
  +\int_0^T\!\!\int_\T m^2
  -\int_0^T|\meanx{m}|^2\,dt,
\]
whereas the convective flux contributes
\[
 -\int_0^T\!\!\int_\T \rho u^2(\rho-M)
 =-\int_0^T\!\!\int_\T\rho^2u^2
   +M\int_0^T\!\!\int_\T\rho u^2.
\]
Since $m^2=\rho^2u^2$, the two terms which need not be integrable separately
cancel algebraically.  The remaining terms give \eqref{eq:primitive-id}.

To obtain the estimate, first note that
$\|w(t)\|_\infty\le\|\rho(t)-M\|_1\le2M$.  The boundary term and the two
convective terms in \eqref{eq:primitive-id} are controlled by mass and kinetic
energy.  The electromagnetic term satisfies
\[
 \int_0^T\!\!\int_\T |jbw|
 \le 2M\|j\|_{L^2_{t,x}}
       \|b\|_{L^2_tL^2_x}.
\]
Finally, for every $\kappa>0$, Young's inequality gives
\[
 \mu\int_0^T\!\!\int_\T |u_x|\rho
 \le \kappa\int_0^T\!\!\int_\T\rho^{\gamma+1}
 +C_{\kappa,T}\int_0^T\!\!\int_\T
      |u_x|^{(\gamma+1)/\gamma}.
\]
The last integral is controlled by the $L^2$ dissipation because
$(\gamma+1)/\gamma<2$.  On the left of \eqref{eq:primitive-id},
\[
 a\rho^\gamma(\rho-M)=a\rho^{\gamma+1}-aM\rho^\gamma,
\]
and the negative part is controlled by the internal energy.  Choosing
$\kappa<a/2$ proves \eqref{eq:pressure-gain}.  Each additional nonnegative
monomial is treated in the same way.
\end{proof}

\section{Weak-coefficient compactness for Maxwell--Ohm}
\label{sec:maxwell}

We first establish the energy well-posedness and stability estimates for the
Maxwell--Ohm subsystem.

For a prescribed coefficient $v\in L^2(0,T;H^1(\T))$, we consider
\begin{equation}
\label{MOwithinitial}
 \eps E_t-b_x=-\sigma(E+vb),\qquad
 b_t-E_x=0,\qquad (E,b)|_{t=0}=(E_0,b_0).
\end{equation}
At the natural energy level one controls $E$ and $b$ in $L^2_x$, but not
their spatial derivatives.  Thus \eqref{MOwithinitial} need not define an
ordinary differential equation with values in $L^2(\T)^2$.  We therefore use
the semigroup formulation associated with the constant-coefficient Maxwell
operator and call the resulting mild (variation-of-constants) solution the
\emph{energy solution}.  Lemma~\ref{lem:maxwell-energy} constructs this mild
solution, proves the energy bounds below, and, through its Fourier--Galerkin
approximation, also shows that it satisfies \eqref{MOwithinitial} in
spacetime distributions.

Its current $j=\sigma(E+vb)$ belongs to $L^2_{t,x}$: indeed,
$H^1(\T)\hookrightarrow L^\infty(\T)$ gives
\[
 \|vb\|_{L^2_{t,x}}
 \le \|v\|_{L^2_tL^\infty_x}\|b\|_{L^\infty_tL^2_x}<\infty,
\]
while $E\in C([0,T];L^2(\T))\subset L^2_{t,x}$.

\begin{lemma}[Energy well-posedness]\label{lem:maxwell-energy}
For every $v\in L^2_tH^1_x$ and
$X_0=(E_0,b_0)\in L^2(\T)^2$, system \eqref{MOwithinitial}
has a unique energy solution
\[
 X=(E,b)^{\mathsf T}\in C([0,T];L^2(\T)^2).
\]
Moreover, on bounded subsets of
\[
 L^2(0,T;H^1(\T))\times L^2(\T)^2,
\]
the solution satisfies
\begin{equation}
 \|X\|_{C_tL^2_x}^2
 +\|E\|_{L^2_{t,x}}^2
 +\|j\|_{L^2_{t,x}}^2
 \le C.
 \label{eq:maxwell-energy-bound}
\end{equation}
If $X_v$ and $X_z$ have coefficients $v,z$ and initial data
$X_{v,0},X_{z,0}$, respectively, then
\begin{equation}
 \|X_v-X_z\|_{C_tL^2_x}^2
 \le C\bigl(
   \|X_{v,0}-X_{z,0}\|_2^2
   +\|v-z\|_{L^2_tL^\infty_x}^2
 \bigr).
 \label{eq:maxwell-Lip}
\end{equation}
Here $C$ may be chosen to depend only on $T,\eps,\sigma$ and upper bounds for
$\|v\|_{L^2_tH^1_x}$, $\|z\|_{L^2_tH^1_x}$,
$\|X_{v,0}\|_{L^2_x}$, and $\|X_{z,0}\|_{L^2_x}$.
\end{lemma}

\begin{proof}
Use the weighted energy space
\[
 \mathcal H_\eps=L^2(\T)^2,\qquad
 X=\binom E b,\qquad
 \|X\|_{\mathcal H_\eps}^2=\eps\|E\|_2^2+\|b\|_2^2,
\]
and the operators
\[
 A_\sigma\binom E b
   =\binom{\eps^{-1}(b_x-\sigma E)}{E_x},\qquad
 D(A_\sigma)=H^1(\T)^2,\qquad
 \mathcal B_f\binom E b=\binom{-\sigma\eps^{-1}fb}{0}.
\]
Thus \eqref{MOwithinitial} is $$X_t=A_\sigma X+\mathcal B_vX.$$
Periodic integration by parts gives
$\langle A_\sigma X,X\rangle_{\mathcal H_\eps}=-\sigma\|E\|_2^2$.
For $\lambda>0$, the Fourier symbol of $\lambda-A_\sigma$ at
$\kappa_k=2\pi k$ is
\[
 \begin{pmatrix}\lambda+\sigma/\eps&-i\kappa_k/\eps\\
                 -i\kappa_k&\lambda\end{pmatrix},\qquad
 D_k=\lambda(\lambda+\sigma/\eps)+\kappa_k^2/\eps>0.
\]
The inverse matrix has entries bounded by $C_\lambda(1+|k|)^{-1}$,
so it maps $L^2(\T)^2$ into $H^1(\T)^2$.  Hence
$\operatorname{Ran}(\lambda-A_\sigma)=\mathcal H_\eps$; the densely
defined dissipative operator $A_\sigma$ is maximal dissipative and generates
the contraction semigroup $S(t)$ on $\mathcal H_\eps$.  In particular,
\begin{equation}
 \|S(t)Y\|_{\mathcal H_\eps}\le\|Y\|_{\mathcal H_\eps}
 \quad(t\ge0,\ Y\in\mathcal H_\eps).
 \label{eq:semigroup-contraction}
\end{equation}
For $f\in L^\infty(\T)$,
\begin{equation}
 \|\mathcal B_fX\|_{\mathcal H_\eps}^2
 =\sigma^2\eps^{-1}\|fb\|_2^2
 \le\sigma^2\eps^{-1}\|f\|_\infty^2\|X\|_{\mathcal H_\eps}^2.
 \label{eq:Bv-operator-bound}
\end{equation}
Since $v\in L^2_tH^1_x\hookrightarrow L^2_tL^\infty_x$, the function
$\beta_v(t)=\|\mathcal B_{v(t)}\|_{\mathcal L(\mathcal H_\eps)}$
belongs to $L^1(0,T)$.  The variation-of-constants equation is
\begin{equation}
 X(t)=S(t)X_0+\int_0^tS(t-s)\mathcal B_{v(s)}X(s)\,ds.
 \label{eq:maxwell-volterra}
\end{equation}
On each subinterval $I$ with $\int_I\beta_v<1/2$, its right-hand side is
a contraction on $C(I;\mathcal H_\eps)$ with the prescribed starting
state.  A finite partition into such intervals constructs the unique mild
solution on $[0,T]$; uniqueness also follows directly from Gronwall's
inequality.

To justify the energy calculation and distributional interpretation, let
$P_N$ be the componentwise Fourier projection onto $|k|\le N$.  It is
orthogonal in $\mathcal H_\eps$ and commutes with $A_\sigma$ and $S(t)$.
Solve
\[
 (X_N)_t=A_\sigma X_N+P_N\mathcal B_vX_N,\qquad
 X_N(0)=P_NX_0.
\]
Subtracting the mild equations and using Gronwall gives
\begin{equation}
 \|X_N-X\|_{C_t\mathcal H_\eps}
 \le e^{\|\beta_v\|_1}
 \bigl(\|(P_N-I)X_0\|_{\mathcal H_\eps}
       +\|(P_N-I)\mathcal B_vX\|_{L^1_t\mathcal H_\eps}\bigr)
 \longrightarrow0.
 \label{eq:galerkin-to-mild}
\end{equation}
The last term tends to zero by dominated convergence, since
$\mathcal B_vX\in L^1_t\mathcal H_\eps$.
For a smooth test, integration by parts in the projected equations and
$P_N\Psi\to\Psi$ in $H^1$ now give \eqref{MOwithinitial} in distributions,
including its initial data.  Conversely, a $C_t\mathcal H_\eps$
distributional solution has $\mathcal B_vX\in L^1_t\mathcal H_\eps$.
Testing against smooth spatial functions and using density in
$D(A_\sigma^*)=H^1(\T)^2$ gives the weak adjoint formulation; Ball's
variation-of-constants theorem~\cite{Ball1977} identifies it with
\eqref{eq:maxwell-volterra}.

Because $P_NX_N=X_N$, orthogonality gives
$\langle P_N\mathcal B_vX_N,X_N\rangle_{\mathcal H_\eps}
 =\langle\mathcal B_vX_N,X_N\rangle_{\mathcal H_\eps}$.
The Maxwell derivative terms cancel under periodic integration, leaving
\[
 \frac12\frac{d}{dt}\|X_N\|_{\mathcal H_\eps}^2
 +\sigma\|E_N\|_2^2=-\sigma\int_\T vb_NE_N\,dx.
\]
Convergence \eqref{eq:galerkin-to-mild} passes this identity to $X$ in
integrated form.  Young's inequality then yields
\begin{equation}
 \frac{d}{dt}\|X\|_{\mathcal H_\eps}^2+\sigma\|E\|_2^2
 \le\sigma\|v\|_\infty^2\|X\|_{\mathcal H_\eps}^2,
 \label{eq:maxwell-energy-differential}
\end{equation}
and consequently
\begin{equation}
 \sup_{t\le T}\|X(t)\|_{\mathcal H_\eps}^2
 \le\|X_0\|_{\mathcal H_\eps}^2
       e^{\sigma\|v\|_{L^2_tL^\infty_x}^2}.
 \label{eq:maxwell-sup-bound}
\end{equation}
Integration bounds $E$ in $L^2_{t,x}$, and
\[
 \|j\|_{L^2_{t,x}}^2
 \le2\sigma^2\bigl(\|E\|_{L^2_{t,x}}^2
       +\|v\|_{L^2_tL^\infty_x}^2\|b\|_{L^\infty_tL^2_x}^2\bigr)
\]
proves \eqref{eq:maxwell-energy-bound}.

For the coefficient stability, put $Y=X_v-X_z=(e,h)^{\mathsf T}$.
The same Galerkin calculation, followed by \eqref{eq:galerkin-to-mild},
justifies
\[
 \frac12\frac{d}{dt}\|Y\|_{\mathcal H_\eps}^2+\sigma\|e\|_2^2
 =-\sigma\int_\T vhe\,dx-\sigma\int_\T(v-z)b_ze\,dx.
\]
Here the Galerkin fields with coefficient $z$ obey
\eqref{eq:maxwell-sup-bound} uniformly in their projection index.
Young's inequality gives
\[
 \frac{d}{dt}\|Y\|_{\mathcal H_\eps}^2+\sigma\|e\|_2^2
 \le C\sigma\|v\|_\infty^2\|Y\|_{\mathcal H_\eps}^2
     +C\sigma\|v-z\|_\infty^2\|b_z\|_2^2.
\]
Gronwall, \eqref{eq:maxwell-sup-bound} for $X_z$, and equivalence of the
weighted and ordinary product $L^2$ norms prove \eqref{eq:maxwell-Lip}.
All constants depend only on the bounds listed in the lemma.  In
particular, they do not depend on any Fourier cutoff used to approximate
the coefficient.
\end{proof}

\begin{proof}[Proof of \cref{thm:maxwell}]
Write $X_n=(E_n,b_n)^{\mathsf T}$, $X=(E,b)^{\mathsf T}$ and
$X_{0,n}=(E_{0,n},b_{0,n})^{\mathsf T}$, $X_0=(E_0,b_0)^{\mathsf T}$.
Use $\mathcal H_\eps$, $S(t)$ and $\mathcal B_f$ from
\cref{lem:maxwell-energy}.

\smallskip
\noindent\emph{Step 1: uniform approximation of the coefficients.}
Let $P_\Lambda$ be the spatial Fourier projection onto $|k|\le\Lambda$.
For $f\in H^1(\T)$ and $\Lambda\ge1$,
\begin{equation}
 \|(I-P_\Lambda)f\|_\infty
 \le C\Lambda^{-1/2}\|f_x\|_2.
 \label{eq:coef-tail}
\end{equation}
Indeed, Cauchy--Schwarz gives
\[
 \sum_{|k|>\Lambda}|\widehat f(k)|
 \le\left(\sum_{|k|>\Lambda}|k|^{-2}\right)^{1/2}
     \left(\sum_{|k|>\Lambda}|k|^2|\widehat f(k)|^2\right)^{1/2}
 \le C\Lambda^{-1/2}\|f_x\|_2.
\]
Weak convergence bounds $v_n$ in $L^2_tH^1_x$, and $P_\Lambda$ is a
contraction on $H^1$.  Let $X_n^\Lambda$ and $X^\Lambda$ be the full energy
solutions with coefficients $P_\Lambda v_n$ and $P_\Lambda v$, and initial
data $X_{0,n}$ and $X_0$, respectively.  The coefficient-tail estimate and
\eqref{eq:maxwell-Lip} imply
\begin{equation}
 \sup_n\|X_n-X_n^\Lambda\|_{C_t\mathcal H_\eps}
       +\|X-X^\Lambda\|_{C_t\mathcal H_\eps}
 \le C\Lambda^{-1/2},
 \label{eq:cutoff-stability}
\end{equation}
where $C$ is independent of $n$ and $\Lambda$.  This uniformity follows
from the $H^1$ contraction and the bounded initial states; no uniform
$L^\infty$ operator norm for $P_\Lambda$ is being assumed.

\smallskip
\noindent\emph{
Step 2: convergence at fixed cutoff.}
For fixed $\Lambda$, in the real basis
\[
 \{e_\ell:1\le\ell\le d_\Lambda\}
 =\{1,\cos(2\pi kx),\sin(2\pi kx):1\le k\le\Lambda\},
 \qquad d_\Lambda=2\Lambda+1,
\]
write
\[
 P_\Lambda v_n=\sum_{\ell=1}^{d_\Lambda}c_{n,\ell}(t)e_\ell,
 \qquad
 P_\Lambda v=\sum_{\ell=1}^{d_\Lambda}c_\ell(t)e_\ell,
 \qquad B_\ell=\mathcal B_{e_\ell}.
\]
Each $B_\ell$ is bounded on $\mathcal H_\eps$ by
\eqref{eq:Bv-operator-bound}, and
\begin{equation}
 c_{n,\ell}\rightharpoonup c_\ell\quad\text{in }L^2(0,T).
 \label{eq:finite-mode-coeff-weak}
\end{equation}
This also implies weak convergence in $L^1(0,T)$, since
$L^\infty(0,T)\subset L^2(0,T)$.  Thus the truncated-coefficient equation is
a finite-input bilinear evolution
\[
 (X_n^\Lambda)_t=A_\sigma X_n^\Lambda
       +\sum_{\ell=1}^{d_\Lambda}c_{n,\ell}(t)B_\ell X_n^\Lambda
\]
on the full infinite-dimensional energy space.  Only the coefficient has
been truncated; there is no state Galerkin projection in this step.

We  now prove the required finite-sum convergence directly, using the standard frozen-trajectory argument; see \cite{BallMarsdenSlemrod1982,Dirr2023} for related weak-control continuity arguments. Put
\[
 F_{n,\ell}(t)=\int_0^t
 (c_{n,\ell}-c_\ell)(s)S(t-s)B_\ell X^\Lambda(s)\,ds.
\]
The kernel $K_\ell(t,s)=S(t-s)B_\ell X^\Lambda(s)$ is continuous on
$0\le s\le t\le T$: a strongly continuous semigroup acts jointly
continuously on time and state, and $X^\Lambda$ is continuous.
For each fixed $t$, uniformly approximate the continuous
$\mathcal H_\eps$-valued function $K_\ell(t,\cdot)$ on $[0,t]$ by a step
function.  Its integral against $c_{n,\ell}-c_\ell$ tends to zero by
\eqref{eq:finite-mode-coeff-weak}, while the approximation error is bounded
by its uniform error times $\|c_{n,\ell}-c_\ell\|_{L^1}$.
Hence $F_{n,\ell}(t)\to0$ in $\mathcal H_\eps$ for every $t$.
For $0\le r<t\le T$, split the difference of the integrals at $r$.
Uniform continuity of $K_\ell$ on the triangle and Cauchy--Schwarz give
\[
 \|F_{n,\ell}(t)-F_{n,\ell}(r)\|_{\mathcal H_\eps}
 \le C_\Lambda\omega_\ell(|t-r|)+C_\Lambda|t-r|^{1/2},
\]
where $\omega_\ell(h)\to0$ as $h\to0$, independently of $n$.
Pointwise convergence and this common modulus, applied to a finite time
net, yield $\|F_{n,\ell}\|_{C_t\mathcal H_\eps}\to0$.

Subtracting the two mild equations, and retaining the initial-data
variation, now gives
\[
 \begin{split}
 D_n(t):=X_n^\Lambda(t)-X^\Lambda(t)
 ={}&S(t)(X_{0,n}-X_0)+\sum_{\ell=1}^{d_\Lambda}F_{n,\ell}(t)\\
 &+\sum_{\ell=1}^{d_\Lambda}\int_0^t
   c_{n,\ell}(s)S(t-s)B_\ell D_n(s)\,ds.
 \end{split}
\]
Consequently,
\[
 \|D_n\|_{C_t\mathcal H_\eps}
 \le\left(\|X_{0,n}-X_0\|_{\mathcal H_\eps}
       +\sum_{\ell=1}^{d_\Lambda}\|F_{n,\ell}\|_{C_t\mathcal H_\eps}\right)
       \exp\!\left(\sum_{\ell=1}^{d_\Lambda}
          \|B_\ell\|\,\|c_{n,\ell}\|_{L^1}\right)
 \longrightarrow0.
\]
At fixed $\Lambda$ the exponential is uniformly bounded in $n$.
This proves convergence of the entire sequence with strongly varying initial
states.  The argument uses a finite sum and does not require the operators
$B_\ell$ to commute.

\smallskip
\noindent\emph{Step 3: removal of the cutoff and product closure.}
The triangle inequality gives
\[
 \|X_n-X\|_{C_t\mathcal H_\eps}
 \le\|X_n-X_n^\Lambda\|_{C_t\mathcal H_\eps}
     +\|X_n^\Lambda-X^\Lambda\|_{C_t\mathcal H_\eps}
     +\|X^\Lambda-X\|_{C_t\mathcal H_\eps}.
\]
First let $n\to\infty$ with $\Lambda$ fixed; then let $\Lambda\to\infty$
using \eqref{eq:cutoff-stability}.  This proves
\eqref{eq:maxwell-strong} with no cutoff-dependent constant entering the
last limit.

It remains to identify the current and Lorentz force.  The energy estimate
gives a uniform $L^2_{t,x}$ bound for $j_n$, and
\[
 v_nb_n-vb=v_n(b_n-b)+(v_n-v)b.
\]
The first term tends to zero strongly in $L^2_{t,x}$ because
\[
 \|v_n(b_n-b)\|_{L^2_{t,x}}
 \le\|v_n\|_{L^2_tL^\infty_x}
      \|b_n-b\|_{L^\infty_tL^2_x}.
\]
To treat the second term, let $\Phi\in L^2_{t,x}$.  The 1D
embedding $H^1(\T)\hookrightarrow L^\infty(\T)$ implies by duality that
$L^1(\T)\hookrightarrow H^{-1}(\T)$, and therefore
\[
 \|b(t)\Phi(t)\|_{H^{-1}_x}
 \le C\|b(t)\Phi(t)\|_{L^1_x}
 \le C\|b(t)\|_2\|\Phi(t)\|_2.
\]
It follows that
\[
 \|b\Phi\|_{L^2_tH^{-1}_x}
 \le C\|b\|_{L^\infty_tL^2_x}\|\Phi\|_{L^2_{t,x}}.
\]
Weak convergence of $v_n$ in $L^2_tH^1_x$ now gives
\[
 \int_0^T\!\!\int_\T(v_n-v)b\Phi\,dx\,dt\longrightarrow0.
\]
Thus $v_nb_n\rightharpoonup vb$ in $L^2_{t,x}$, and Ohm's law yields
\[
 j_n=\sigma(E_n+v_nb_n)\rightharpoonup\sigma(E+vb)=j
 \quad\text{in }L^2_{t,x}.
\]
Finally, for every $\Phi\in C_c^\infty((0,T)\times\T)$,
\[
 \int(j_nb_n-jb)\Phi
 =\int j_n(b_n-b)\Phi+\int(j_n-j)b\Phi.
\]
The first term tends to zero because $j_n$ is bounded in $L^2_{t,x}$ and
$b_n-b\to0$ strongly in $L^2_{t,x}$; the second tends to zero by the weak
convergence of $j_n$, since $b\Phi\in L^2_{t,x}$.  This proves
\eqref{eq:lorentz-stable}.
\end{proof}

\section{The global finite-energy construction}
\label{sec:application}

We prove \cref{thm:main} by inserting \cref{thm:maxwell} into the standard
artificial-viscosity and artificial-pressure construction for compressible
Navier--Stokes equations
\cite{Lions1998,FeireislNovotnyPetzeltova2001}.  The electromagnetic closure,
density gain, and pressure identification are interdependent parts of this
construction.  We give the coupling-specific details for the exact Maxwell evolution
at the Galerkin level, electromagnetic closure in each fluid limit, the
Lorentz force against moving density primitives, the uniform density gain,
and elimination of the pressure defect.  Routine fluid compactness,
renormalization, and initial-trace steps are recalled only to fix notation and
are otherwise taken from the cited construction.

\subsection{Approximation, energy, and the Galerkin limit}
We let 
\begin{align}
 H_\eta(z)&=\frac{a}{\gamma-1}z^\gamma+\eta z^2
             +\frac{\eta}{\beta-1}z^\beta,                 \label{eq:Heta}\\
 p_\eta(z)&=zH_\eta'(z)-H_\eta(z)
 =az^\gamma+\eta z^2+\eta z^\beta,                      \label{eq:peta}
\end{align}
where 
\(
 \beta>\max\{4,\gamma\}, \eta>0,\,\delta>0.
\)
For each $\eta$, choose smooth electromagnetic initial data
$(E_{0,\eta},b_{0,\eta})$ converging strongly to $(E_0,b_0)$ in
$L^2(\T)^2$ as $\eta\to0$.  The same pair is used at every $N$- and
$\delta$-level with $\eta$ fixed.  Thus there is no hidden electromagnetic
projection in any of the approximation limits.
We use
\begin{align}
 \rho_t+(\rho u)_x&=\delta\rho_{xx},                       \label{eq:reg-mass}\\
 (\rho u)_t+(\rho u^2)_x+p_\eta(\rho)_x-\mu u_{xx}
 +\delta\rho_xu_x&=-jb,                                   \label{eq:reg-mom}\\
 \eps E_t-b_x&=-j,                                        \label{eq:reg-ampere}\\
 b_t-E_x&=0,                                               \label{eq:reg-faraday}\\
 j&=\sigma(E+ub).                                         \label{eq:reg-ohm}
\end{align}
Thus neither an electromagnetic viscosity nor a Maxwell Galerkin truncation
is introduced.  At the semi-Galerkin level, $u_N$ belongs to the real Fourier
space
\[
 X_N=\operatorname{span}_{\mathbb R}
 \{1,\cos(2\pi kx),\sin(2\pi kx):1\le k\le N\},
\]
the continuity equation is solved as a scalar parabolic equation, and the
complete subsystem
\begin{equation}
 \eps(E_N)_t-(b_N)_x=-j_N,\qquad
 (b_N)_t-(E_N)_x=0,\qquad
 j_N=\sigma(E_N+u_Nb_N)                                  \label{eq:N-Maxwell-system}
\end{equation}
is solved exactly.  Only the momentum equation is projected onto $X_N$.
The following proposition records this coupling and the associated
continuation estimate.

\begin{proposition}[Semi-Galerkin construction with exact Maxwell evolution]
\label{prop:semi-galerkin}
Fix $N\in\mathbb N$ and $\delta,\eta>0$.  Let
$\rho_{0,N}\in C^\infty(\T)$ be strictly positive, let
$u_{0,N}\in X_N$, and take smooth periodic electromagnetic data
$(E_{0,N},b_{0,N})$.  Then the semi-Galerkin problem consisting of
\eqref{eq:reg-mass}, the exact Maxwell--Ohm system
\eqref{eq:N-Maxwell-system}, and
\begin{align}
 \frac{d}{dt}\int_\T\rho_Nu_N\varphi\,dx
 ={}&\int_\T(\rho_Nu_N^2+p_\eta(\rho_N))\varphi_x\,dx
      -\mu\int_\T(u_N)_x\varphi_x\,dx                    \label{eq:N-momentum}\\
 &-\delta\int_\T(\rho_N)_x(u_N)_x\varphi\,dx
      -\int_\T j_Nb_N\varphi\,dx                         \nonumber
\end{align}
for every $\varphi\in X_N$, admits a global solution.  Moreover,
$\rho_N(t,x)>0$, and for every $t\ge0$,
\begin{align}
 &\int_\T\left[
 \frac12\rho_Nu_N^2+H_\eta(\rho_N)
 +\frac\eps2E_N^2+\frac12b_N^2\right](t)\,dx              \label{eq:reg-energy}\\
 &\quad+\int_0^t\!\!\int_\T\left[
 \mu|(u_N)_x|^2+\delta H_\eta''(\rho_N)|(\rho_N)_x|^2
 +\frac1\sigma|j_N|^2\right]dx\,ds                       \nonumber\\
 &=\int_\T\left[
 \frac12\rho_{0,N}u_{0,N}^2+H_\eta(\rho_{0,N})
 +\frac\eps2E_{0,N}^2+\frac12b_{0,N}^2\right]dx.          \nonumber
\end{align}
In particular, the electromagnetic component of the approximation is the
full Maxwell--Ohm energy solution, not a finite-dimensional truncation.
\end{proposition}

\begin{proof}
We indicate the coupling argument and give the complete energy and
continuation calculations.  The remaining local fixed-point construction is
the standard semi-Galerkin procedure for compressible Navier--Stokes
approximations; see \cite{Lions1998,FeireislNovotnyPetzeltova2001}.

For a prescribed path $u_N\in C([0,\tau];X_N)$, the first equation in
\eqref{eq:reg-mass} is uniformly parabolic and has a unique positive smooth
solution.  With the same prescribed coefficient, \cref{lem:maxwell-energy}
gives the unique solution of the full, unprojected Maxwell--Ohm system
\eqref{eq:N-Maxwell-system}.  Its dependence on $u_N$ is continuous by
\eqref{eq:maxwell-Lip}.  Finally, if
$\{\phi_1,\ldots,\phi_{d_N}\}$ is a real basis of $X_N$, the matrix
\[
 \mathsf M_N[\rho_N]_{k\ell}
 =\int_\T\rho_N\phi_k\phi_\ell\,dx
\]
is positive definite.  Equation \eqref{eq:N-momentum} is therefore an
ordinary differential system for the coefficients of $u_N$.  The usual
short-time fixed-point argument closes these three solution maps.  Notice
that no approximation of $(E_N,b_N)$ is used in this construction.

We next derive the a priori identity which globalizes the solution.  Testing
the parabolic continuity equation by $H_\eta'(\rho_N)$ and using
$p_\eta(z)=zH_\eta'(z)-H_\eta(z)$ gives
\begin{equation}
 \frac{d}{dt}\int_\T H_\eta(\rho_N)\,dx
 +\int_\T p_\eta(\rho_N)(u_N)_x\,dx
 +\delta\int_\T H_\eta''(\rho_N)|(\rho_N)_x|^2\,dx=0.
                                                               \label{eq:N-internal}
\end{equation}
Since $u_N$ is an admissible test in \eqref{eq:N-momentum}, the kinetic
energy calculation is exact.  In fact,
\[
 \int_\T u_N\bigl[(\rho_Nu_N)_t+(\rho_Nu_N^2)_x\bigr]dx
 =\frac{d}{dt}\int_\T\frac12\rho_Nu_N^2\,dx
  -\delta\int_\T(\rho_N)_xu_N(u_N)_x\,dx.
\]
The last term is cancelled precisely by the correction
$\delta(\rho_N)_x(u_N)_x$ in \eqref{eq:reg-mom}.  Consequently,
\begin{equation}
 \frac{d}{dt}\int_\T\frac12\rho_Nu_N^2\,dx
 -\int_\T p_\eta(\rho_N)(u_N)_x\,dx
 +\mu\int_\T|(u_N)_x|^2\,dx
 =-\int_\T j_Nb_Nu_N\,dx.                                \label{eq:N-kinetic}
\end{equation}
On the other hand, multiplying the two exact Maxwell equations by $E_N$ and
$b_N$, respectively, and using periodicity yields
\begin{equation}
 \frac{d}{dt}\int_\T\left(\frac\eps2E_N^2+\frac12b_N^2\right)dx
 =-\int_\T j_NE_N\,dx.                                   \label{eq:N-EM-energy}
\end{equation}
Adding \eqref{eq:N-internal}, \eqref{eq:N-kinetic}, and
\eqref{eq:N-EM-energy}, and then using
$j_N=\sigma(E_N+u_Nb_N)$, gives
\[
 -j_N(E_N+u_Nb_N)=-\frac1\sigma|j_N|^2.
\]
Integration in time proves \eqref{eq:reg-energy}.  This calculation also
shows why keeping the Maxwell subsystem exact is compatible with the fluid
Galerkin projection: the test $u_N$ belongs to $X_N$, whereas the Maxwell
energy identity requires no projected test.

For completeness, the standard continuation criterion is also unchanged.
The energy identity and \cref{lem:weighted-poincare} bound
$u_N$ in $L^2_tH^1_x$, while finite dimensionality gives
$\int_0^T\|(u_N)_x\|_\infty\,dt<\infty$.  The parabolic maximum principle
therefore keeps $\rho_N$ strictly positive and bounded on each finite
interval, so the Galerkin mass matrix remains invertible.  All terms in the
finite-dimensional momentum equation are integrable; for the only new one,
\[
 \int_0^T\!\!\int_\T|j_Nb_N|\,dx\,dt
 \le T^{1/2}\|j_N\|_{L^2_{t,x}}
       \|b_N\|_{L^\infty_tL^2_x}<\infty.
\]
The usual Galerkin continuation argument
\cite{Lions1998,FeireislNovotnyPetzeltova2001} now gives the global solution.
\end{proof}

The next lemma is the interface between the compactness theorem and all three
fluid approximation limits.  It will be used with the generic index $n$
replaced successively by $N$, $\delta$, and $\eta$.

\begin{lemma}[Electromagnetic closure along fluid limits]
\label{lem:maxwell-closure}
Let $u_n\rightharpoonup u$ in $L^2(0,T;H^1(\T))$, and suppose
\[
 (E_{0,n},b_{0,n})\longrightarrow(E_0,b_0)
 \quad\text{strongly in }L^2(\T)^2.
\]
For every $n$, let $(E_n,b_n,j_n)$ solve the exact Maxwell--Ohm subsystem
\[
 \eps(E_n)_t-(b_n)_x=-j_n,\qquad
 (b_n)_t-(E_n)_x=0,\qquad
 j_n=\sigma(E_n+u_nb_n).
\]
Then, for the energy solution $(E,b)$ with coefficient $u$ and initial datum
$(E_0,b_0)$,
\begin{align}
 &(E_n,b_n)\longrightarrow(E,b)
     &&\text{strongly in }C([0,T];L^2(\T)^2),               \label{eq:closure-fields}\\
 &u_nb_n\rightharpoonup ub,\qquad
   j_n\rightharpoonup j:=\sigma(E+ub)
     &&\text{weakly in }L^2((0,T)\times\T),                \label{eq:closure-current}\\
 &j_nb_n\longrightarrow jb
     &&\text{in }\mathcal D'((0,T)\times\T).              \label{eq:closure-force}
\end{align}
In particular, the limit satisfies the complete Maxwell--Ohm subsystem,
\[
 \eps E_t-b_x=-j,\qquad b_t-E_x=0,
 \qquad j=\sigma(E+ub),
\]
where the first two equations hold in distributions and Ohm's law holds in
$L^2_{t,x}$.
\end{lemma}

\begin{proof}
The strong convergence \eqref{eq:closure-fields} and the Lorentz-force
convergence \eqref{eq:closure-force} follow from \cref{thm:maxwell}.  The
product convergence $u_nb_n\rightharpoonup ub$ in $L^2_{t,x}$ is the final
product estimate in the proof of that theorem; combining it with
$E_n\to E$ in $C_tL^2_x$ identifies the weak limit of
$j_n=\sigma(E_n+u_nb_n)$ and proves \eqref{eq:closure-current}.  Passing to
the limit in the two linear Maxwell equations gives the displayed
Maxwell--Ohm subsystem.  Uniqueness in \cref{lem:maxwell-energy} identifies
the limit without any further subsequence.
\end{proof}

Distributional convergence of $j_nb_n$ by itself does not allow an
approximation-dependent test function.  The next lemma supplies the precise
moving-test statement needed in the density-primitive identities.

\begin{lemma}[Lorentz force against moving density primitives]
\label{lem:moving-lorentz}
Assume
\[
 u_n\rightharpoonup u\quad\text{in }L^2(0,T;H^1(\T)),\qquad
 (E_n,b_n)\to(E,b)\quad\text{in }C([0,T];L^2(\T)^2),
\]
and let $w_n\to w$ in $C([0,T]\times\T)$.  Let
\(
 j_n=\sigma(E_n+u_nb_n),\,\; j=\sigma(E+ub).
\)
Then
\begin{equation}
\int_\T j_nb_nw_n\,dx
 \longrightarrow
 \int_\T jbw\,dx
 \quad\text{in }\mathcal D'(0,T).
\label{eq:lorentz-moving-test}
\end{equation}

\end{lemma}

\begin{proof}
The current product argument at the end of the proof of \cref{thm:maxwell}
uses only the convergence hypotheses stated here, not the Maxwell equations.
Indeed, $u_n(b_n-b)\to0$ in $L^2_{t,x}$, while, for every
$\Phi\in L^2_{t,x}$, $b\Phi\in L^2_tH^{-1}_x$.  Thus
$(u_n-u)b\rightharpoonup0$ in $L^2_{t,x}$ and $j_n\rightharpoonup j$
in $L^2_{t,x}$.  Also,
\[
 \|b_nw_n-bw\|_{L^2_{t,x}}
 \le \|w_n\|_\infty\|b_n-b\|_{L^2_{t,x}}
     +\|w_n-w\|_\infty\|b\|_{L^2_{t,x}}\longrightarrow0.
\]
For $\chi\in C_c^\infty(0,T)$, it follows that
\[
 \int\chi(j_nb_nw_n-jbw)
 =\int\chi j_n(b_nw_n-bw)+\int\chi(j_n-j)bw\longrightarrow0,
\]
by the strong $L^2$ convergence of the test product and the weak $L^2$
convergence of the current.  This proves \eqref{eq:lorentz-moving-test}.
\end{proof}

At fixed $(\delta,\eta)$, the standard Galerkin compactness argument gives,
after extraction,
\[
 \rho_N\to\rho_{\delta,\eta}\ \hbox{strongly in }L^q_{t,x}
 \ (q<\beta+1),\qquad
 u_N\rightharpoonup u_{\delta,\eta}\ \hbox{in }L^2_tH^1_x,
\]
and $(\rho_N)_x\to(\rho_{\delta,\eta})_x$ strongly in $L^2_{t,x}$.
Compactness of the projected momentum in $L^2_tH^{-1}_x$ identifies the
convective term, while using $P_N\varphi$ as test function removes the
Galerkin projection.  These are precisely the classical steps in
\cite{Lions1998,FeireislNovotnyPetzeltova2001}; the electromagnetic limit is
the only additional point.
At fixed $\eta$, the initial pair $(E_{0,\eta},b_{0,\eta})$ is independent of
$N$ and $\delta$.  Since \eqref{eq:N-Maxwell-system} is the exact
Maxwell--Ohm subsystem, \cref{lem:maxwell-closure} applies and yields
\begin{equation}
 (E_N,b_N)\to(E_{\delta,\eta},b_{\delta,\eta})
       \ \hbox{in }C_tL^2_x,\;\;\;
 j_N\rightharpoonup j_{\delta,\eta}\ \hbox{in }L^2_{t,x},\;\;\;
 j_Nb_N\to j_{\delta,\eta}b_{\delta,\eta}\ \hbox{in }\D'.
                                                                  \label{eq:N-Maxwell}
\end{equation}
Moreover,
\[
 u_Nb_N\rightharpoonup
 u_{\delta,\eta}b_{\delta,\eta}\quad\text{in }L^2_{t,x},
 \qquad
 j_{\delta,\eta}
 =\sigma(E_{\delta,\eta}+u_{\delta,\eta}b_{\delta,\eta}).
\]
Consequently,
\[
 \eps(E_{\delta,\eta})_t-(b_{\delta,\eta})_x=-j_{\delta,\eta},
 \qquad
 (b_{\delta,\eta})_t-(E_{\delta,\eta})_x=0
\]
in distributions.  Thus the entire Maxwell--Ohm subsystem, including its
initial data, is identified at the Galerkin limit.
The remaining terms pass exactly as in the standard fluid construction.
For the Lorentz projection error one uses the uniform $L^1_{t,x}$ bound on
$j_Nb_N$ and $P_N\varphi\to\varphi$ uniformly.  Hence the Galerkin limit
solves \eqref{eq:reg-mass}--\eqref{eq:reg-ohm} at fixed $(\delta,\eta)$.

\subsection{Removal of the density viscosity}

Fix $\eta>0$ and denote the Galerkin limits just constructed by
$(\rho_\delta,u_\delta,E_\delta,b_\delta,j_\delta)$, suppressing $\eta$ from
the notation.  Their electromagnetic initial data are
$(E_{0,\eta},b_{0,\eta})$, independently of $\delta$.

Let $w_\delta=\partial_x^{-1}(\rho_\delta-M)$.  The diffusive continuity
equation gives
\[
 (w_\delta)_t=-\rho_\delta u_\delta
 +\meanx{\rho_\delta u_\delta}+\delta(\rho_\delta)_x.
\]
Testing \eqref{eq:reg-mom} by $w_\delta$ gives the identity of
\cref{lem:primitive-identity}, with $p$ replaced by $p_\eta$, up to
\[
 \mathcal R_\delta
 =-\delta\int_0^T\!\!\int_\T
      \rho_\delta u_\delta(\rho_\delta)_x
   +\delta\int_0^T\!\!\int_\T
      (\rho_\delta)_x(u_\delta)_xw_\delta .
\]
Because $\sqrt{\delta\eta}(\rho_\delta)_x$ is bounded in $L^2$ and
$\|w_\delta\|_\infty\le2M$, we have 
\(
 |\mathcal R_\delta|\le C_{\eta,T}\sqrt\delta.
\)
The primitive estimate therefore yields, uniformly for $0<\delta\le1$,
\begin{equation}
 \int_0^T\!\!\int_\T
 \bigl(\rho_\delta^{\gamma+1}+\eta\rho_\delta^3
       +\eta\rho_\delta^{\beta+1}\bigr)
 \,dx\,dt\le C_{\eta,T}.                                  \label{eq:delta-gain}
\end{equation}
Moreover,
\[
 \delta\rho_{\delta,x}\to0\quad\hbox{in }L^2,\;\;
 \delta\rho_{\delta,x}u_{\delta,x}\to0\quad\hbox{in }L^1.
\]
The full momentum equation now makes $m_\delta=\rho_\delta u_\delta$
compact in $L^2_tH^{-1}_x$, and convection passes as above.
The energy identity and \cref{lem:weighted-poincare} also give, after
extraction,
\[
 u_\delta\rightharpoonup u_\eta
 \quad\text{in }L^2(0,T;H^1(\T)).
\]
The electromagnetic initial data are fixed at this limit.  Hence
\cref{lem:maxwell-closure} with $n=\delta$ gives
\begin{align}
 &(E_\delta,b_\delta)\to(E_\eta,b_\eta)
       &&\text{strongly in }C([0,T];L^2(\T)^2),            \label{eq:delta-Maxwell}\\
 &u_\delta b_\delta\rightharpoonup u_\eta b_\eta
       &&\text{weakly in }L^2_{t,x},                       \nonumber\\
 &j_\delta=\sigma(E_\delta+u_\delta b_\delta)
       \rightharpoonup j_\eta=\sigma(E_\eta+u_\eta b_\eta)
       &&\text{weakly in }L^2_{t,x},                       \nonumber\\
 &j_\delta b_\delta\longrightarrow j_\eta b_\eta
       &&\text{in }\D'((0,T)\times\T).                     \nonumber
\end{align}
Passing to the limit in the two linear equations also gives
\[
 \eps(E_\eta)_t-(b_\eta)_x=-j_\eta,\qquad
 (b_\eta)_t-(E_\eta)_x=0
\]
in distributions.  Thus the $\delta\to0$ limit retains the complete,
unregularized Maxwell--Ohm subsystem.
Moreover, let $w_\eta=\partial_x^{-1}(\rho_\eta-M)$.  The proof of
\cref{lem:primitive} applies to $w_\delta$ with the additional term
$\delta(\rho_\delta)_x$ in $(w_\delta)_t$.  This term tends to zero in
$L^2_{t,x}$ because
\[
\|\delta(\rho_\delta)_x\|_{L^2_{t,x}}
 \le \sqrt{\frac{\delta}{\eta}}\,
     \|\sqrt{\delta\eta}(\rho_\delta)_x\|_{L^2_{t,x}}
 \longrightarrow0.
\]
Let $q=\min\{\gamma,2\}>1$.  The nondiffusive part of
$(w_\delta)_t$ is bounded in $L^2(0,T;L^\gamma(\T))$, while the additional
term is bounded in $L^2_{t,x}$ and tends to zero there.  Since $|\T|=1$,
$(w_\delta)_t$ is therefore bounded in $L^2(0,T;L^q(\T))$, and
\[
 \|w_\delta(t)-w_\delta(s)\|_{L^q_x}
 \le C|t-s|^{1/2}.
\]
Together with the uniform $L^\infty(0,T;W^{1,\gamma}(\T))$ bound, the same
1D Gagliardo--Nirenberg and Arzel\`a--Ascoli argument as in
\cref{lem:primitive} gives
\[
 w_\delta\to w_\eta\quad\text{in }C([0,T]\times\T).
\]
All hypotheses of \cref{lem:moving-lorentz} now follow from
\eqref{eq:delta-Maxwell} and the weak convergence of $u_\delta$.  Hence
\begin{equation}
 \int_\T j_\delta b_\delta w_\delta\,dx
 \longrightarrow
 \int_\T j_\eta b_\eta w_\eta\,dx
 \quad\text{in }\mathcal D'(0,T).                          \label{eq:delta-moving-lorentz}
\end{equation}
Thus the Lorentz coupling passes in the time-localized primitive identity
without any appeal to $L^1$ equi-integrability.

Once \eqref{eq:delta-moving-lorentz} identifies the only new term, pressure
identification is the standard effective-flux argument.  Subtracting the
approximate and limiting primitive identities gives the usual nonnegative
monotonicity defect, while the diffusive and limiting
$\rho\log\rho$ identities give the opposite sign.  Hence the defect vanishes,
\[
 \rho_\delta\to\rho_\eta\quad\hbox{a.e. and strongly in }L^q_{t,x}
 \quad(q<\beta+1),
\]
and $p_\eta(\rho_\delta)\to p_\eta(\rho_\eta)$ in $L^1$.  We refer to
\cite{Lions1998,FeireislNovotnyPetzeltova2001} for this unchanged final
step.  Thus the limit solves the system with pressure $p_\eta$ and without
density viscosity.

The estimate \eqref{eq:delta-gain} has a constant depending on $\eta$, so it
cannot be used directly in the limit $\eta\to0$.  Once the density viscosity
has been removed, however, the primitive identity has no diffusion remainder
and yields the required uniform estimate.

\begin{proposition}[Uniform density gain after removal of density viscosity]
\label{prop:eta-gain}
For the solutions $(\rho_\eta,u_\eta,E_\eta,b_\eta,j_\eta)$ constructed above,
and for every $T>0$,
\begin{equation}
 \int_0^T\!\!\int_\T
 \bigl(\rho_\eta^{\gamma+1}+\eta\rho_\eta^3
       +\eta\rho_\eta^{\beta+1}\bigr)\,dx\,dt\le C_T,     \label{eq:eta-gain}
\end{equation}
where $C_T$ is independent of $\eta\in(0,1]$.
\end{proposition}

\begin{proof}
Fix $\eta>0$ and set
$w_\eta=\partial_x^{-1}(\rho_\eta-M)$.  Lower semicontinuity in the
$\delta\to0$ limit and \eqref{eq:delta-gain} first give the corresponding
bound with an $\eta$-dependent constant.  In particular,
$p_\eta(\rho_\eta)\rho_\eta\in L^1_{t,x}$.  Moreover, the artificial-pressure
energy gives $\rho_\eta\in L^\infty_tL^\beta_x$, and hence
\[
 \rho_\eta^2u_\eta^2\in L^1((0,T)\times\T)
 \qquad(\beta>4).
\]
Thus all terms in the density-primitive calculation are integrable at fixed
$\eta$.  Approximating $w_\eta$ in space-time and using
\[
 (w_\eta)_x=\rho_\eta-M,\qquad
 (w_\eta)_t=-\rho_\eta u_\eta
                 +\meanx{\rho_\eta u_\eta},
\]
justifies the calculation of \cref{lem:primitive-identity} for this weak
solution.  In particular,
\begin{align}
 \int_0^T\!\!\int_\T p_\eta(\rho_\eta)(\rho_\eta-M)
 ={}&\left[\int_\T \rho_\eta u_\eta w_\eta\,dx\right]_0^T
 +\int_0^T\left(M\int_\T\rho_\eta u_\eta^2\,dx
      -\left|\int_\T\rho_\eta u_\eta\,dx\right|^2\right)dt \notag\\
 &+\mu\int_0^T\!\!\int_\T (u_\eta)_x(\rho_\eta-M)\,dx\,dt
 +\int_0^T\!\!\int_\T j_\eta b_\eta w_\eta\,dx\,dt.     \label{eq:eta-primitive-id}
\end{align}
The two copies of $\rho_\eta^2u_\eta^2$ cancel before any estimate is made,
exactly as in \cref{lem:primitive-identity}.

The initial approximations specified in \cref{subsec:initial-data} have
uniformly bounded augmented energy.  We now estimate
\eqref{eq:eta-primitive-id} using only quantities controlled uniformly by that
energy.  Since $\|w_\eta\|_{L^\infty_{t,x}}\le2M$, the boundary
and convective terms are bounded by the mass and kinetic energy.  Also,
\[
 \int_0^T\!\!\int_\T |j_\eta b_\eta w_\eta|\,dx\,dt
 \le 2M\|j_\eta\|_{L^2_{t,x}}
          \|b_\eta\|_{L^2_tL^2_x}\le C_T.
\]
For every $\kappa>0$, Young's inequality and
$(\gamma+1)/\gamma<2$ give
\[
 \mu\int_0^T\!\!\int_\T |(u_\eta)_x|(\rho_\eta+M)\,dx\,dt
 \le \kappa\int_0^T\!\!\int_\T\rho_\eta^{\gamma+1}\,dx\,dt
      +C_{\kappa,T}\bigl(1+\|(u_\eta)_x\|_{L^2_{t,x}}^2\bigr).
\]
Finally,
\begin{align*}
 p_\eta(\rho_\eta)(\rho_\eta-M)
 ={}&a\rho_\eta^{\gamma+1}+\eta\rho_\eta^3
       +\eta\rho_\eta^{\beta+1}\\
 &-M\bigl(a\rho_\eta^\gamma+\eta\rho_\eta^2
       +\eta\rho_\eta^\beta\bigr).
\end{align*}
The last line is  uniformly bounded in \(L^1((0,T)\times\mathbb T)\).
Choosing $\kappa<a/2$ and absorbing the first term on the right proves
\eqref{eq:eta-gain} with a constant independent of $\eta$.
\end{proof}

\subsection{Removal of the artificial pressure}

Let
\(
(\rho_\eta,u_\eta,E_\eta,b_\eta,j_\eta)
\)
be the solution family at the final approximation level. The energy, \cref{lem:weighted-poincare,prop:eta-gain} give, uniformly in
$\eta$,
\begin{align}
 &\rho_\eta\ \hbox{bounded in }L^\infty_tL^\gamma_x,\qquad
 u_\eta\ \hbox{bounded in }L^2_tH^1_x,                    \label{eq:eta-basic}\\
 &E_\eta,b_\eta\ \hbox{bounded in }L^\infty_tL^2_x,\qquad
 j_\eta\ \hbox{bounded in }L^2_{t,x},                     \nonumber\\
 &\int_0^T\!\!\int_\T
 \bigl(\rho_\eta^{\gamma+1}+\eta\rho_\eta^3
 +\eta\rho_\eta^{\beta+1}\bigr)\,dx\,dt\le C_T.          \nonumber
\end{align}
In particular,
\begin{equation}
 \eta\rho_\eta^2\to0\quad\hbox{in }L^{3/2},\;\;\;
 \eta\rho_\eta^\beta\to0
 \quad\hbox{in }L^{(\beta+1)/\beta}.                      \label{eq:art-p-zero}
\end{equation}
Writing $p_0=2\gamma/(\gamma+1)>1$, the kinetic and density energies yield
$m_\eta=\rho_\eta u_\eta$ bounded in
$L^\infty_tL^{p_0}_x$.  The momentum equation bounds
$\partial_tm_\eta$ in $L^1_tW^{-1,1}_x$.  Since
$L^{p_0}(\T)\Subset H^{-1}(\T)$, Simon's theorem gives
\begin{equation}
 m_\eta\to\rho u\quad\hbox{strongly in }L^2_tH^{-1}_x,
 \qquad
 \rho_\eta u_\eta^2\to\rho u^2\quad\hbox{in }\D'.         \label{eq:eta-momentum}
\end{equation}
In addition, \eqref{eq:eta-basic} and
\cref{lem:weighted-poincare} give, after extraction,
\[
 u_\eta\rightharpoonup u
 \quad\text{in }L^2(0,T;H^1(\T)).
\]
By construction,
$(E_{0,\eta},b_{0,\eta})\to(E_0,b_0)$ strongly in $L^2(\T)^2$.  Therefore
\cref{lem:maxwell-closure}, now with $n=\eta$, gives
\begin{equation}
(E_\eta,b_\eta)\to(E,b)
\qquad\text{strongly in }C([0,T];L^2(\T)^2).
\label{eq:eta-Maxwell}
\end{equation}

Moreover,
\[
u_\eta b_\eta\rightharpoonup ub,
\qquad
j_\eta=\sigma(E_\eta+u_\eta b_\eta)
\rightharpoonup j=\sigma(E+ub)
\quad\text{weakly in }L^2_{t,x},
\]
and
\[
j_\eta b_\eta\longrightarrow jb
\qquad\text{in }\D'((0,T)\times\T).
\]

These are the electromagnetic closure inputs needed in addition to the
compressible-fluid compactness argument in the limit $\eta\to0$.
In particular, the physical limit satisfies
\begin{equation}
 \eps E_t-b_x=-j,\qquad b_t-E_x=0,\qquad
 j=\sigma(E+ub)                                           \label{eq:limit-Maxwell}
\end{equation}
 with the first two identities in distributions and Ohm's law in
$L^2_{t,x}$.  Let $w=\partial_x^{-1}(\rho-M)$.  The exact continuity equations
and \cref{lem:primitive} give
\[
 w_\eta\to w\quad\text{in }C([0,T]\times\T).
\]
Applying \cref{lem:moving-lorentz} with $n=\eta$, together with
\eqref{eq:eta-Maxwell} and the weak convergence of $u_\eta$, yields
\begin{equation}
 \int_\T j_\eta b_\eta w_\eta\,dx
 \longrightarrow
 \int_\T jbw\,dx
 \quad\text{in }\mathcal D'(0,T).                          \label{eq:eta-moving-lorentz}
\end{equation}
This is the moving-test convergence required when the primitive identity is
passed to the physical limit.

The last step shows explicitly that the electromagnetic term leaves no
residual defect in the effective-flux argument.

\begin{proposition}[Elimination of the pressure defects]
\label{prop:pressure-defect}
Let $(\rho_\eta,u_\eta,E_\eta,b_\eta,j_\eta)$ be the family above.  Then,
after extraction, for every finite $T$,
\begin{equation}
 \rho_\eta\longrightarrow\rho\quad\hbox{strongly in }
 L^{\gamma+1}((0,T)\times\T),\qquad
 \rho_\eta^\gamma\longrightarrow\rho^\gamma
 \quad\hbox{strongly in }L^{(\gamma+1)/\gamma}.           \label{eq:eta-density}
\end{equation}
Moreover, the concentration measures generated by
$\eta\rho_\eta^3$ and $\eta\rho_\eta^{\beta+1}$ vanish.  Consequently,
\[
 p_\eta(\rho_\eta)\longrightarrow a\rho^\gamma
 \quad\hbox{strongly in }L^1((0,T)\times\T).
\]
\end{proposition}

\begin{proof}
After replacing the usual inverse-divergence test by the 1D density primitive
introduced above, the pressure identification is the standard
Lions--Feireisl effective-flux and renormalization argument
\cite{Lions1998,FeireislNovotnyPetzeltova2001}.  We therefore record only
the additional electromagnetic term that must be controlled in passing to the
limit.

Let
\[
 w_\eta=\partial_x^{-1}(\rho_\eta-M),\qquad
 w=\partial_x^{-1}(\rho-M).
\]
By \cref{lem:primitive},
$w_\eta\to w$ in $C([0,T]\times\T)$, while the electromagnetic closure gives
\[
 b_\eta\to b
 \quad\hbox{strongly in }C([0,T];L^2(\T)),\qquad
 j_\eta\rightharpoonup j
 \quad\hbox{weakly in }L^2((0,T)\times\T).
\]
Hence
\[
 b_\eta w_\eta\longrightarrow bw
 \quad\hbox{strongly in }L^2((0,T)\times\T),
\]
and therefore
\[
 \int_\T j_\eta b_\eta w_\eta\,dx
 \longrightarrow
 \int_\T jbw\,dx
 \quad\hbox{in }\D'(0,T),
\]
which is precisely \eqref{eq:eta-moving-lorentz}.  Thus the Lorentz term
passes to the limit in the time-localized density-primitive identity.

It also produces no concentration at the initial endpoint.  Indeed, the energy
bounds and $\|w_\eta\|_{L^\infty_{t,x}}\le 2M$ give, uniformly in $\eta$,
\[
 \int_0^\tau\!\int_\T |j_\eta b_\eta w_\eta|\,dx\,dt
 \le
 \|w_\eta\|_{L^\infty_{t,x}}
 \|j_\eta\|_{L^2((0,\tau)\times\T)}
 \|b_\eta\|_{L^2((0,\tau)\times\T)}
 \le C\tau^{1/2}.
\]
Consequently the electromagnetic contribution creates no additional defect
measure in the effective-flux identity, including at $t=0$.

All remaining terms are handled by the classical monotonicity,
effective-viscous-flux, and renormalized-continuity argument
\cite{Lions1998,FeireislNovotnyPetzeltova2001}.
In the present setting, applying that argument to the time-localized
density-primitive identity, together with the renormalized
$\rho\log\rho$ identity, gives the following defect inequality.
Writing
\[
 Q_T=(0,T)\times\T,
\]
one obtains
\[
 \limsup_{\eta\to0}
 \int_{Q_T}
 \left(
   a\rho_\eta^{\gamma+1}
   +\eta\rho_\eta^3
   +\eta\rho_\eta^{\beta+1}
 \right)\,dx\,dt
 \le
 a\int_{Q_T}\rho^{\gamma+1}\,dx\,dt.
\]
The only additional term relative to the classical compressible-fluid
argument is the Lorentz contribution.  Its passage to the limit is
provided by \eqref{eq:eta-moving-lorentz}, while the estimate above shows
that it creates no concentration at the initial endpoint.

On the other hand, \eqref{eq:eta-gain} gives boundedness of
$\rho_\eta$ in $L^{\gamma+1}(Q_T)$, and the weak limit in this space is
$\rho$.  Hence weak lower semicontinuity yields
\[
 a\int_{Q_T}\rho^{\gamma+1}\,dx\,dt
 \le
 a\liminf_{\eta\to0}
 \int_{Q_T}\rho_\eta^{\gamma+1}\,dx\,dt.
\]
Combining the last two inequalities and using the nonnegativity of the
artificial-pressure terms gives
\[
 \int_{Q_T}\rho_\eta^{\gamma+1}\,dx\,dt
 \longrightarrow
 \int_{Q_T}\rho^{\gamma+1}\,dx\,dt,
\]
and
\[
 \int_{Q_T}\eta\rho_\eta^3\,dx\,dt
 +
 \int_{Q_T}\eta\rho_\eta^{\beta+1}\,dx\,dt
 \longrightarrow0.
\]
Thus the nonnegative concentration measures generated by
$\eta\rho_\eta^3$ and $\eta\rho_\eta^{\beta+1}$ vanish.
Since $L^{\gamma+1}(Q_T)$ is uniformly convex, weak convergence together
with convergence of the norms gives
\[
 \rho_\eta\longrightarrow\rho
 \quad\text{strongly in }L^{\gamma+1}(Q_T).
\]
Consequently,
\[
 \rho_\eta^\gamma\longrightarrow\rho^\gamma
 \quad\text{strongly in }L^{(\gamma+1)/\gamma}(Q_T).
\]
Finally, combining this with \eqref{eq:art-p-zero} yields
\[
 p_\eta(\rho_\eta)\longrightarrow a\rho^\gamma
 \quad\text{strongly in }L^1(Q_T),
\]
and proves \eqref{eq:eta-density}.

\end{proof}

\subsection{Initial data and completion of the proof}
\label{subsec:initial-data}

The fluid initial data are approximated by the standard finite-energy
construction for compressible Navier--Stokes equations
\cite{Lions1998,FeireislNovotnyPetzeltova2001}, with smooth positive
densities of mass $M$ such that
$\rho_{0,\eta}\to\rho_0$ in $L^\gamma(\T)$, momenta
$m_{0,\eta}\to m_0$ in $L^{p_0}(\T)$, and convergent augmented initial
energy.

For the electromagnetic variables, choose smooth periodic data
\[
 (E_{0,\eta},b_{0,\eta})\longrightarrow(E_0,b_0)
 \qquad\text{strongly in }L^2(\T)^2,
\]
and use the same pair at every $N$- and $\delta$-level for fixed $\eta$.
Consequently, the strong convergences in
\eqref{eq:N-Maxwell}, \eqref{eq:delta-Maxwell}, and
\eqref{eq:eta-Maxwell} hold uniformly down to $t=0$, so the limiting
electromagnetic fields attain the prescribed initial data $(E_0,b_0)$
strongly in $L^2(\T)^2$.

Convex lower semicontinuity in $(\rho,m)$, strong electromagnetic
convergence, and weak lower semicontinuity of $u_x$ and $j$ pass the
approximate energy identity to \eqref{eq:energy-ineq}.  The weak formulations
follow from the convergences established above, while the renormalized
continuity equation follows from the standard commutator argument
\cite{DiPernaLions1989}.  Since the approximate solutions are global and the preceding estimates hold on
every finite time interval, the compactness argument above yields a global
finite-energy weak solution and completes the proof of \cref{thm:main}.

\section*{Acknowledgments}

The author is partially supported by NSF grant DMS-2510425 and by the Simons
Foundation grant MPS-TSM-00007824.

\section*{AI Use Disclosure}

During the preparation of this work, the author used ChatGPT (OpenAI) and Claude (Anthropic) to assist with language editing, checking notation and internal consistency, and LaTeX preparation. The author reviewed and edited any AI-generated output and takes full responsibility for the content of the article.

\end{document}